\documentclass[a4paper,11pt]{amsart}
\usepackage{cmap}
\usepackage[T1]{fontenc}
\usepackage[utf8]{inputenc}
\usepackage[english]{babel}
\usepackage{amsfonts,amssymb,amsthm,amsmath}
\usepackage{url}
\usepackage{enumitem}

\usepackage{secdot}

\usepackage{graphicx}
\usepackage{epstopdf}
\usepackage{varwidth}

\usepackage{a4wide}
\usepackage{xcolor}
\usepackage[colorlinks=true,linktocpage,pdfpagelabels,
bookmarksnumbered,bookmarksopen]{hyperref}
\definecolor{ForestGreen}{rgb}{0.1,0.6,0.05}
\definecolor{EgyptBlue}{rgb}{0.063,0.1,0.6}
\hypersetup{
	colorlinks=true,
	linkcolor=EgyptBlue,         
	citecolor=ForestGreen,
	urlcolor=olive
}

\usepackage[hyperpageref]{backref}

\let\OLDthebibliography\thebibliography
\renewcommand\thebibliography[1]{
	\OLDthebibliography{#1}
	\setlength{\parskip}{1pt}
	\setlength{\itemsep}{1pt plus 0.3ex}
}

\DeclareRobustCommand\nlam{\lambda}
\DeclareRobustCommand\nom{\Omega}

\numberwithin{equation}{section}
\newtheorem{theorem}{Theorem}[section]

\newtheorem{cor}[theorem]{Corollary}
\theoremstyle{definition}
\newtheorem{remark}[theorem]{Remark}

\title[On partially free boundary solutions]{On partially free boundary solutions for elliptic problems with non-Lipschitz nonlinearities}

\author[V.~Bobkov]{V.~Bobkov}
\author[P.~Dr\'abek]{P.~Dr\'abek}
\author[Y.~Ilyasov]{Y.~Ilyasov}

\address[V.~Bobkov, P.~Dr\'abek]{\newline\indent
	Department of Mathematics and NTIS, Faculty of Applied Sciences,
	\newline\indent 
	University of West Bohemia, Univerzitn\'i 8, 301 00 Plze\v{n}, Czech Republic
}
\email{bobkov@kma.zcu.cz, pdrabek@kma.zcu.cz}

\address[Y.~Ilyasov]{\newline\indent
	Institute of Mathematics, Ufa Federal Research Centre, RAS,
	\newline\indent 
	Chernyshevsky str. 112, 450008 Ufa, Russia
	\newline\indent
	Instituto de Matem\'atica e Estat\'istica, Universidade Federal de Goi\'as
	\newline\indent 
	74001-970, Goiania, Brazil
}
\email{ilyasov02@gmail.com}

\thanks{
	V.~Bobkov and P.~Dr\'abek were supported by the grant 18-03253S of the Grant Agency of the Czech Republic. V. Bobkov was also supported by the project LO1506 of the Czech Ministry of Education, Youth and Sports.
	Y.~Ilyasov wishes to thank the University of West Bohemia, where this research was conducted, for the invitation and hospitality.
	The authors would like to thank J.~I.\ D\'iaz and J.\ Hern\'andez for the encouraging and stimulating discussions.
}

\begin{document}

\begin{abstract}
	We show that the elliptic equation with a non-Lipschitz right-hand side, $-\Delta u = \lambda |u|^{\beta-1}u - |u|^{\alpha-1} u$ with $\lambda>0$ and $0<\alpha<\beta<1$, considered on a smooth star-shaped domain $\Omega$ subject to zero Dirichlet boundary conditions, might possess a nonnegative ground state solution which violates Hopf's maximum principle only on a nonempty subset $\Gamma$ of the boundary $\partial\Omega$ such that $\Gamma \neq \partial\Omega$.
\end{abstract}	
\maketitle

\section{Introduction and main result}

Let $\Omega \subset \mathbb{R}^N$ be a smooth bounded domain, $N \geq 2$.
Consider the boundary value problem
\begin{equation}\label{D}
\tag*{$\mathcal{P}(\nlam,\nom)$}
\left\{
\begin{aligned}
-\Delta u &= \lambda |u|^{\beta-1}u - |u|^{\alpha-1} u \quad \text{in } \Omega,\\
u &= 0 \quad \text{on } \partial\Omega,
\end{aligned}
\right.
\end{equation}
where $\lambda>0$ and $0< \alpha < \beta < 1$, that is, the nonlinearity in \ref{D} is non-Lipschitz at zero. 
The latter property prevents to conclude that nonnegative solutions of \ref{D} \textit{a priori} obey the strong maximum principle or Hopf's maximum principle (the boundary point lemma). 
In fact, as it follows from \cite{Kaper1} (see also \cite{Diaz,IlyasovEgorov}), there exists $\lambda^*>0$ such that  problem {\renewcommand\nlam{\lambda^*}\ref{D}} possesses  the so-called \textit{free boundary solution} (equivalently, \textit{compact support solution}), which is a nonzero solution $u$ such that 
\begin{equation}\label{eq:zero}
\frac{\partial u}{\partial \nu} = 0 \text{ on } \partial \Omega,
\end{equation}
where $\nu$ is the unit outward normal vector to $\partial \Omega$.

Let us note that a symmetry result of \cite{SZ} (see also \cite{Kaper2}) together with a uniqueness result of \cite{Kaper1} force any connected component of the support of a nonnegative free boundary solution of \ref{D} to be a ball whose radius $R_\lambda$ is uniquely defined by $\lambda$. 
Moreover, $\lambda \mapsto R_\lambda$ is a decreasing function, as it follows from a simple scaling argument. 
Therefore, it is clear that for any $\lambda>\lambda^*$ problem \ref{D} has a continuum of free boundary solutions. 

At the same time, in \cite{HMV} there was proved the existence of $\underline{\lambda}>0$ such that for any $\lambda>\underline{\lambda}$ problem \ref{D} has a positive solution $u$ which \textit{does} satisfy Hopf's maximum principle at every point of the boundary $\partial \Omega$, i.e., 
\begin{equation*}\label{eq:nonzero}
\frac{\partial u}{\partial \nu} < 0 \text{ on } \partial \Omega.
\end{equation*}

The existence of two types of solutions described above naturally leads to the problem of the existence of a complementary class of solutions. Namely, we consider the following question:

\begin{quote}
Is there a solution of problem \ref{D} which violates Hopf's maximum principle only on a part of the boundary $\partial \Omega$?
\end{quote}

The aim of our note is to give an affirmative answer to this question. More precisely, we are interested in finding so-called \textit{partially free boundary solution} of \ref{D}, i.e., a solution $u$ which satisfies
\begin{equation}\label{eq:viol}
\frac{\partial u}{\partial \nu} = 0 \text{ on } \Gamma
\quad \text{and} \quad 
\frac{\partial u}{\partial \nu} \neq 0 \text{ on } \partial \Omega \setminus \Gamma
\end{equation}
for some nonempty subset $\Gamma \subsetneq \partial\Omega$.
Let us state our main result.
\begin{theorem}\label{thm}
	Let $N$, $\alpha$, $\beta$ satisfy 
	\begin{equation}\label{eq:abN}
	N \geq 3,
	\quad 
	0<\alpha<\beta<1,
	\quad
	2(1+\alpha)(1+\beta) - N(1-\alpha)(1-\beta) < 0.
	\end{equation} 	
	Then there exist a bounded, strictly star-shaped (i.e., $(x,\nu) > 0$ on $\partial \Omega$) domain $\Omega \subset \mathbb{R}^N$ of class $C^2$ and a value $\lambda>0$  such that problem \ref{D} possesses a nonnegative partially free boundary ground state solution.
\end{theorem}

Here, under a \textit{ground state solution} of \ref{D} we mean a solution $u$ with the least action property, namely, $E_\lambda(u) \leq E_\lambda(v)$ for any nonzero solution $v$ of \ref{D}, where $E_\lambda$ is the energy functional associated with \ref{D}:
$$
E_\lambda(u) = \frac{1}{2} \int_\Omega |\nabla u|^2 \,dx - \frac{\lambda}{\beta+1} \int_\Omega |u|^{\beta+1} \, dx + \frac{1}{\alpha+1} \int_\Omega |u|^{\alpha+1} \,dx.
$$

\begin{remark}
	The result of \cite{DGM} on the existence of radial sign-changing solutions in combination with the compact support principle \cite{PSZ} implies that problem \ref{D}, considered on an annulus $\Omega = B_{R_2} \setminus \overline{B_{R_1}}$, has for certain values of $R_1$, $R_2$, $R_2>R_1$, and $\lambda$ a radial positive solution $u$ such that 
	$$
	\frac{\partial u}{\partial \nu} = 0 \text{ on } \partial B_{R_2}
	\quad \text{and} \quad 
	\frac{\partial u}{\partial \nu} < 0 \text{ on } \partial B_{R_1}.
	$$
	That is, $u$ is a partially free boundary solution. However, apart from this example, the existence of other partially free boundary solutions of \ref{D}, including the case in which $\Omega$ is a simply connected domain, was not known to the authors.
	
	Here and below, $B_{R} \subset \mathbb{R}^N$ stands for the open ball of radius $R$ centred at the origin. 
\end{remark}

The proof of Theorem \ref{thm} is given in Section \ref{sec:construction}. We provide an explicit construction of the domain with required properties. 
To this end, in Section \ref{sec:prelim}, we recall some auxiliary results based on investigations made in \cite{Diaz}.

\section{Preliminaries}\label{sec:prelim}
Our approach to prove Theorem \ref{thm} relies on the idea of the proof of the compact support principle \cite[Theorem 2]{PSZ} and the following result obtained in \cite{Diaz}.
\begin{theorem}\label{thm1}
	Let $N, \alpha, \beta$ satisfy 	\eqref{eq:abN}.	
	Let $\Omega \subset \mathbb{R}^N$ be a bounded strictly star-shaped domain of class $C^2$. 
	Then there exists $\lambda^*=\lambda^*(\Omega)\in (0,+\infty)$ such that for any $\lambda \geq \lambda^*$ problem \ref{D} has a nonnegative ground state solution $u_\lambda$, whereas for $\lambda < \lambda^*$, \ref{D} has no nonzero solutions. Moreover, 
	\begin{enumerate}[label={\rm(\roman*)}]
		\item\label{th1:1} for any $\lambda > \lambda^*$ there holds
		\begin{equation}\label{eq:int>0}
		\int_{\partial\Omega} \left|\frac{\partial u_\lambda}{\partial \nu}\right|^{2} \left(x, \nu\right) \, ds > 0,
		\end{equation}
		\item\label{th1:2} $u_{\lambda^*}$ is a free boundary solution, $\text{supp}(u_{\lambda^*}) = \overline{B_{R(\Omega)}}$ is a maximal ball inscribed in $\Omega$, and $u_{\lambda^*}$ is radially symmetric with respect to the centre of $B_{R(\Omega)}$,
		\item\label{th1:3} any sequence $\{u_{\lambda_n}\}$, where $\lambda_n \downarrow \lambda^*$, converges (up to a subsequence) strongly in $H_0^1(\Omega)$ to some $u_{\lambda^*}$.
	\end{enumerate}
\end{theorem}
The existence and nonexistence parts of Theorem \ref{thm1} were obtained in \cite[Theorems 1.1]{Diaz}.
The assertion \ref{th1:1} follows from the following inequality for the Pohozaev identity:
\begin{equation}\label{eq:Poh}
P_\lambda(u_\lambda) 
:= E_\lambda(u_\lambda) - \frac{1}{N} \int_\Omega |\nabla u_\lambda|^2 \,dx = -\frac{1}{2N} \int_{\partial\Omega} \left|\frac{\partial u_\lambda}{\partial \nu}\right|^{2} \left(x, \nu\right) \, ds < 0
\end{equation}
for any $\lambda > \lambda^*$, see \cite[Proposition 4.2 and Corollary 5.3]{Diaz}.
The assertion \ref{th1:2} is stated in \cite[Theorem 1.2]{Diaz}, and the assertion \ref{th1:3} follows from \cite[Lemma 8.1]{Diaz}.

\begin{remark}\label{rem:R}
	Theorem \ref{thm1} \ref{th1:2} together with the uniqueness result of \cite{Kaper1} imply that $\lambda^*(\Omega) = \lambda^*(R(\Omega))$. That is, $\lambda^*$ depends only on the radius of a maximal ball inscribed in $\Omega$.
\end{remark}

\begin{remark}
	Let $\lambda>\lambda^*$. Evidently, \eqref{eq:int>0} means that $u_\lambda$ cannot be a free boundary solution. 
	At the same time, \eqref{eq:int>0} does not provide more detailed information about the pointwise behaviour of $\frac{\partial u_\lambda}{\partial \nu}$ on $\partial \Omega$. 
	In particular, it is not known \textit{a priori} whether $u_\lambda$ is either a partially free boundary solution or it satisfies Hopf's maximum principle on the whole of $\partial\Omega$.
\end{remark}

In Section \ref{sec:construction} below, we will also need the following refinement of Theorem \ref{thm1} \ref{th1:3}.
\begin{cor}\label{lem:convergence}
	Under the assumptions of Theorem \ref{thm1}, any sequence $\{u_{\lambda_n}\}$, where $\lambda_n \downarrow \lambda^*$, converges (up to a subsequence) in $C^1(\overline{\Omega})$ to some $u_{\lambda^*}$.
\end{cor}
\begin{proof}
	Since $\lambda^* \in (0,+\infty)$ and $u_{\lambda_n} \to u_{\lambda^*}$ strongly in $H_0^{1}(\Omega)$ up to a subsequence, we see from Theorem \ref{thm1} \ref{th1:3} that $\{u_{\lambda_n}\}$ is bounded in $H_0^{1}(\Omega)$. Therefore, applying the standard bootstrap argument (see, e.g., \cite[Lemma 3.2, p.\ 114]{DKN}), we obtain that $\{u_{\lambda_n}\}$ is bounded in $L^\infty(\Omega)$. Thus, the regularity result \cite{Lieberman} implies that $\{u_{\lambda_n}\}$ is bounded in $C^{1,\gamma}(\overline{\Omega})$ for some $\gamma \in (0,1)$. Finally, the Arzel\`a-Ascoli theorem yields $u_{\lambda_n} \to u_{\lambda^*}$ in $C^{1}(\overline{\Omega})$ up to a subsequence. 
\end{proof}

\section{Construction}\label{sec:construction}
We will prove Theorem \ref{thm} in three steps.

\textit{Step 1.} Taking the ball $B_1$ of radius $1$, we define the value $\lambda^* = \lambda^*(1)$, see Remark \ref{rem:R}.
Let us fix any $\lambda_0>\lambda^*$ and $\delta > 0$ such that 
\begin{equation}\label{eq:d}
\frac{s^{\alpha+1}}{\alpha+1} - \frac{\lambda_0 s^{\beta+1}}{\beta+1} > 0
\quad \text{for all } s \in (0,\delta).
\end{equation}
Let us also fix some $l_0 > 1$. 
We are going to show the existence of $l_1 > l_0$ and a supersolution $v$ of {\renewcommand\nlam{\lambda_0}\renewcommand\nom{\mathbb{R}^N \setminus \overline{B_{l_0}}}\ref{D}} which has the following properties:
\begin{enumerate}[label={\rm(\roman*)}]
	\item\label{item:1} $v$ is radial;
	\item\label{item:2} $v$ is nonnegative and nonincreasing;
	\item\label{item:3} $v = \delta$ on $\partial B_{l_0}$;
	\item\label{item:4} $v(x) = 0$ for all $x$ satisfying $|x| \geq l_1$. 	
\end{enumerate}
The existence of such $v$ was obtained in the proof of \cite[Theorem 2]{PSZ} in more general settings.
We repeat some arguments from \cite{PSZ} applied to our particular case, for the sake of completeness. Let us define a constant
$$
C = \frac{1}{\sqrt{2}} \int_0^\delta \left(\frac{s^{\alpha+1}}{\alpha+1} - \frac{\lambda_0 s^{\beta+1}}{\beta+1}\right)^{-\frac{1}{2}} ds.
$$
Note that $C<+\infty$ due to the assumption $0<\alpha<\beta<1$ and the choice of $\delta$.
Consider a function $w = w(r)$ for $r \in [0, C)$ given by the implicit formula
\begin{equation}\label{eq:r}
r = \frac{1}{\sqrt{2}} \int_{w(r)}^\delta \left(\frac{s^{\alpha+1}}{\alpha+1} - \frac{\lambda_0 s^{\beta+1}}{\beta+1}\right)^{-\frac{1}{2}} ds.
\end{equation}
Differentiating this equality, we get 
\begin{equation}\label{eq:w'}
-w'(r) = \sqrt{2}  \left(\frac{w(r)^{\alpha+1}}{\alpha+1} - \frac{\lambda_0 w(r)^{\beta+1}}{\beta+1}\right)^{\frac{1}{2}}
\end{equation}
for all $r \in [0,C)$. Thus, we deduce from \eqref{eq:d}, \eqref{eq:r}, and \eqref{eq:w'} that $w'<0$, $w \in (0, \delta]$,  $w(0)=\delta$, and both $w(r) \to 0$ and $w'(r) \to 0$ as $r \uparrow C$. Moreover, 
\begin{equation}\label{eq:radial}
-w'' = \lambda_0 |w|^{\beta-1} w - |w|^{\alpha-1} w
\end{equation}
on the interval $r \in [0,C)$, as it follows from \cite[Lemma 1 (ii)]{PSZ}.
Recalling that $w(r), w'(r) \to 0$ as $r \uparrow C$, we can extend $w$ by zero for $r \geq C$, and hence $w$ is a solution of \eqref{eq:radial} for all $r \geq 0$.

Let us now define $v(x) = w(|x| - l_0)$ for $x \in \mathbb{R}^N \setminus \overline{B_{l_0}}$. Using \eqref{eq:radial}, we obtain
\begin{equation*}
-\Delta v - \left(\lambda_0 |v|^{\beta-1} v - |v|^{\alpha-1} v\right)
=
-w'' - \frac{N-1}{|x|} w' - \left(\lambda_0 |w|^{\beta-1} w - |w|^{\alpha-1} w\right) 
=
- \frac{N-1}{|x|} w'
\geq 0,
\end{equation*}
since $w' \leq 0$.
Thus, $v$ is a supersolution of {\renewcommand\nlam{\lambda_0}\renewcommand\nom{\mathbb{R}^N \setminus \overline{B_{l_0}}}\ref{D}} with the desired properties \ref{item:1}-\ref{item:4} stated above, where $l_1 = l_0 + C$.
Moreover, since $v$ is nonnegative, we also have 
$$
-\Delta v \geq \lambda_0 |v|^{\beta-1}v - |v|^{\alpha-1}v \geq \lambda|v|^{\beta-1}v - |v|^{\alpha-1}v
\quad \text{for all } \lambda \leq \lambda_0.
$$ 
That is, $v$ is a supersolution of {\renewcommand\nom{\mathbb{R}^N \setminus \overline{B_{l_0}}}\ref{D}} for all $\lambda \leq \lambda_0$.

\textit{Step 2.} Let us fix any $l>l_1$ and define a strictly star-shaped domain $\Omega$ as $\Omega = B_1 \cup C_l \cup V_l$ (see Figure \ref{fig1}), where 
\begin{itemize}
	\item[--] $C_l = D_r \times (0, l)$ is an $N$-dimensional cylinder, $D_r \subset \mathbb{R}^{N-1}$ is the open ball of radius $r \in (0,1)$ centred at the origin;
	\item[--] $V_l$ is an appropriate residual part which smooths the boundary of $B_1 \cup C_l$ in such a way that $\Omega$ is strictly star-shaped and its boundary is of class $C^{2}$.
\end{itemize}

\begin{figure}[ht]
	\centering
	\includegraphics[width=0.8\linewidth]{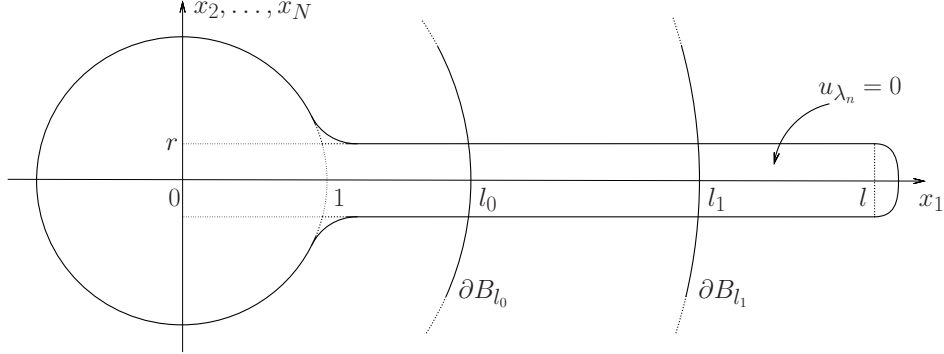}\\
	\caption{The domain $\Omega = B_1 \cup C_l \cup V_l$.}
	\label{fig1}
\end{figure}

Note that, by construction, the radius $R(\Omega)$ of a maximal ball inscribed in $\Omega$ equals $1$. Consequently, $\lambda^*(R(\Omega)) = \lambda^*(1)$ coincides with the value $\lambda^*$ defined above.
Thus, for any $\lambda > \lambda^*$ there exists a corresponding ground state solution $u_{\lambda}$ of {\renewcommand\nom{\Omega}\ref{D}} obtained by Theorem \ref{thm1}.
Now, taking an arbitrary sequence $\lambda_n \downarrow \lambda^*$, we get from Corollary \ref{lem:convergence} that $u_{\lambda_n} \to u_{\lambda^*}$ in $C^1(\overline{\Omega})$ as $n \to +\infty$, up to a subsequence.  
Therefore, recalling that $u_{\lambda^*} = 0$ in $\Omega \setminus B_1$ (see Theorem \ref{thm1} \ref{th1:2}) and $l_0>1$, we deduce the existence of $n_{0} > 0$ such that 
\begin{eqnarray}\label{eq:u<d}
u_{\lambda_n} < \delta~~\mbox{ in } T_{l_0},~~~~\forall n \geq n_{0}.
\end{eqnarray}
Here, $T_{l_0} := \Omega \setminus \overline{B_{l_0}}$ is the ``tail'' of $\Omega$ lying outside of $\overline{B_{l_0}}$. 

\textit{Step 3.} Let us now compare $u_{\lambda_n}$ and the supersolution $v$. From \eqref{eq:u<d} and the property \ref{item:3} of $v$ we have $u_{\lambda_n} \leq \delta = v$ on $\partial T_{l_0} \cap \partial B_{l_0}$ whenever $n \geq n_{0}$. 
Thus, recalling that $u_{\lambda_n}=0$ on $\partial T_{l_0} \setminus \partial B_{l_0}$ and $v$ is nonnegative (see \ref{item:2}), we conclude that  $u_{\lambda_n} \leq v$ on $\partial T_{l_0}$. 
Moreover, since $\lambda_n \downarrow \lambda^*$, we can take $n$ larger, if necessary, to get $\lambda_n \leq \lambda_0$. 
Therefore, in view of \eqref{eq:u<d} and the fact that $v$ is a supersolution of {\renewcommand\nlam{\lambda_n}\renewcommand\nom{T_{l_0}}\ref{D}}, we can apply the weak comparison principle stated in \cite[Lemma 3]{PSZ} to deduce that $u_{\lambda_n} \leq v$ in $T_{l_0}$, which yields $u_{\lambda_n} = 0$ in $\Omega \setminus B_{l_1}$ due to \ref{item:4}. 
In particular, there exists a nonempty subset $\Gamma_n \subsetneq \partial\Omega$ such that $\frac{\partial u_{\lambda_n}}{\partial \nu} = 0$ on $\Gamma_n$. 
Since \eqref{eq:int>0} is satisfied for all $\lambda_n$, we conclude that  
$u_{\lambda_n}$ is a nonnegative partially free boundary ground state solution of {\renewcommand\nlam{\lambda_n}\renewcommand\nom{\Omega}\ref{D}}. This completes the proof of Theorem \ref{thm}.

\medskip
\begin{remark}
	Clearly, the construction of $\Omega$ from above can be substantially generalized. 
	We conjecture that Theorem \ref{thm} holds true for \textit{any} sufficiently smooth bounded strictly star-shaped domain $\Omega$, except maybe the case when $\Omega$ is a ball. 
	More precisely, based on the results of \cite{HMV,An2} and Theorem \ref{thm1}, we conjecture that for any such $\Omega$  nonnegative ground state solutions $u_\lambda$ of \ref{D} have the following behaviour with respect to $\lambda$: there exists $\underline{\lambda}>\lambda^*$ such that 
	\begin{enumerate}[label={\rm(\roman*)}]
		\item $u_\lambda$ is a partially free boundary solution for any $\lambda \in (\lambda^*, \underline{\lambda})$;
		\item $u_{\underline{\lambda}}$ is positive;
		\item $u_\lambda$ is positive and satisfies Hopf's maximum principle on the whole of $\partial \Omega$ for any $\lambda > \underline{\lambda}$.
	\end{enumerate}
	We refer the interested reader to \cite{An1,DHI2} for additional discussions on problem \ref{D}.
\end{remark}

\end{document}